\documentclass{amsart}
\usepackage{chngcntr}
\usepackage{apptools}
\usepackage{color}
\usepackage{amsthm,amssymb,verbatim}
\usepackage{graphicx}
\usepackage{enumerate}
\usepackage{chngcntr}
\usepackage{apptools}
\usepackage{color}
\usepackage{amsthm,amssymb,verbatim}
\usepackage{graphicx}
\usepackage{enumerate}
\newcommand{\pdf}[2]{\frac{\partial #1}{\partial #2}}
\newcommand{\Ricci}{\operatorname{Ric}}   

\newcommand{\Ff}{\mathcal F}

 \newcommand{\Cc}{\mathcal C}

 \newcommand{\RR}{\mathbf{R}}  
 \newcommand{\BB}{\mathbf{B}}  
  \newcommand{\HH}{\textnormal{H}}  
    
  \newcommand{\Div}{\operatorname{Div}}
  
    \newcommand{\dist}{\operatorname{dist}}

 \newcommand{\eps}{\epsilon}

 \newcommand{\tM}{\widetilde{M}}

\newcommand{\vv}{\mathbf v}

\newcommand{\nn}{\mathbf{n}}

\newcommand{\spt}{\operatorname{spt}}

\usepackage{amsthm}

\newcommand{\interior}{\operatorname{interior}}

\newcommand{\fup}{\widehat{f}}
\newcommand{\Kup}{\widehat{K}}
\newcommand{\tf}{\widetilde{f}}
\newcommand{\tK}{\widetilde{K}}

\input epsf
\def\begfig {
\begin{figure}
\small }
\def\endfig {
\normalsize
\end{figure}
}

    \newtheorem{theorem}    {Theorem}   
    \newtheorem{lemma}      [theorem]       {Lemma}
    \newtheorem{corollary}  [theorem]     {Corollary}
    
    \newtheorem{proposition}       [theorem]       {Proposition}

    \newtheorem*{theorem*}{Theorem}
    \theoremstyle{definition}
    \newtheorem{definition}  [theorem] {Definition}
     
    \theoremstyle{definition}
    \newtheorem{remark}   [theorem]       {Remark}
    \newtheorem*{remark*}{Remark}

\title[Avoidance]{Avoidance for Set-Theoretic Solutions 
 of Mean-Curvature-Type Flows}
\author{Or Hershkovits}
\thanks{The first author was partially supported by an AMS-Simons Travel Grant}
\address{Institute of Mathematics\\ Hebrew University \\ Givat Ram \\ Jerusalem, 91904, Israel}
\email{or.Hershkovits@mail.huji.ac.il}
\author{Brian White}
\thanks{The second author was partially supported by NSF grants~DMS-1404282 and~DMS-1711293}
\address{Department of Mathematics\\ Stanford University\\ Stanford, CA 94305}
\email{bcwhite@stanford.edu}
\subjclass[2010]{Primary 53C44; Secondary 49Q20.}
\date{September 13, 2018.  Revised March 23, 2020.}
\usepackage{hyperref}
\usepackage{enumerate}
\usepackage[alphabetic, msc-links]{amsrefs}   
\begin{document}
\maketitle
\begin{abstract}
 We give a self-contained treatment of set-theoretic subsolutions to flow by mean curvature, or, more generally, to flow by mean curvature plus an ambient vector field.  The ambient space can be any smooth Riemannian manifold.
 Most importantly, we show that if
 two such set-theoretic subsolutions are initially disjoint, then they remain disjoint
 provided one of the subsolutions is compact; previously, this was only known for Euclidean space
 (with no ambient vectorfield). 
We also give a simple proof of a version of Ilmanen's interpolation theorem.
\end{abstract}

\section{introduction}
Under mean curvature flow, an initially smooth compact hypersurface in $\RR^{n+1}$ 
must become singular in finite time.   
Singularities typically occur before the surface disappears, that is, before its area tends to zero.
Thus it is desirable to have weak notions of mean curvature flow that allow
the flow to extend past singularities.

{\bf Level set flow}, introduced simultaneously in~\cite{CGG} and~\cite{ES}, is one such notion.
It is very natural and has proved to be very useful.  Under mild hypotheses on the ambient space,
there is a unique level set flow starting with {\bf any} compact initial set; for a smoothly embedded
initial surface, it agrees with the classical solution as long as the classical solution exists (i.e., up
until the first singular time).
However, the definition has the unfortunate feature that a limit of level set flows need not be a level set flow.

Partly to get around that feature, Ilmanen~\cite{Ilm_mani_surv, Ilm_elip}
 introduced a weaker notion, that of a ``{set-theoretic subsolution}
to mean curvature flow'' or (in the terminology of~\cite{White_top_weak}) a ``weak set flow".
Roughly speaking, a one-parameter family of closed subsets of a Riemannian manifold is a weak set flow
provided it does not bump into any smoothly embedded, closed hypersurface moving by mean curvature flow.

A key feature of weak set flows is that not only do they not bump into smooth mean curvature flows,
they also cannot bump into other weak set flows. More precisely, they satisfy the following avoidance
principle: two initially disjoint weak set flows remain disjoint as long as at least one of them remains compact.
(Under the mild hypothesis that the ambient space is complete with Ricci curvature bounded below,
any initially compact weak set flow remains compact.)   Ilmanen gave a very elegant proof of the avoidance
principle in Euclidean space, but it strongly relied on invariance of mean curvature flow under spatial translations,
and thus it did not seem to extend to other Riemannian manifolds.
One of the main contributions of this paper is modifying Ilmanen's proof so that it works
in arbitrary Riemannian manifolds, and, more generally, for closed sets (in a Riemannian manifold)
 moving by mean curvature plus
an ambient vectorfield.

Weak set flows and level set flows are related by a containment theorem (Theorem~\ref{containment-theorem}):
the level set flow starting from a given set is a weak set flow, and it contains every other weak set flow
starting from that set.   
Ilmanen~\cite{Ilm_mani_surv}*{4H} proved that the containment theorem follows from the avoidance principle.
But since the avoidance principle was only known in Euclidean space, likewise the
containment theorem was only known in that case.

The organization of this paper as follows.
Section~\ref{definitions-section} gives the basic definitions.  
We have found it convenient to use a definition of weak set flow that differs from,
but is equivalent to, Ilmanen's original definition.
 In Section \ref{elementary-section}, we derive some elementary properties of weak set flows. 
 In Section \ref{barrier-modification-section}, 
 we prove some technical results about modifying barriers to get barriers with additional
 desirable properties.
In Sections~\ref{distance-function-section} and~\ref{avoidance-section}, 
the barrier modification theorems are used to prove the avoidance principle.
In Section \ref{equivalent-definitions-section}, 
we show that our definition of weak set flow (Definition~\ref{weak-set-flow-definition}) agrees
with Ilmanen's original definition. 
 In Section \ref{biggest-flow-section}, we show that there is a biggest weak set flow with
 any given initial set, and we prove (under mild hypotheses) that this biggest flow
coincides with the level set flow.
In Sections~\ref{limits-section} and~\ref{boundary-section}, we show that limits of weak set flows and boundaries of level set flows are weak set flows. 
In Section \ref{X-flow-section}, we explain how the discussion in this paper extends to motion by mean curvature plus an ambient vectorfield. 
In Section~\ref{X-mean-convex-section}, we present the basic facts about surfaces that move in one direction under the flow.
 In Section~\ref{Brakke_sec}, we consider varifolds flowing by mean curvature plus an ambient vectorfield, 
and we show that the support of such a varifold flow is a weak set flow. 
In the appendix, we give a simple proof of a version of Ilmanen's interpolation theorem, a key tool in 
the proof of the avoidance theorem.

\section{Basic Definitions}\label{definitions-section}

\begin{definition}\label{barrier-definition}
Let $N$ be a smooth Riemannian manifold. 
A family $t\in [a,b]\mapsto K(t)$ of closed subsets of $N$
is called a {\bf smooth barrier} in $N$ provided it is a 
 smooth, one-parameter family of closed regions
with smooth boundary.   
Equivalently, it is a smooth barrier provided there 
exists a smooth function $f:N\times[a,b]\to\RR$ such that $K(t)=\{x: f(x,t)\le 0\}$
and such that $\nabla f(x,t)$ is nonzero at all points of $\partial K(t)$.
We say that the barrier is {\bf compact} if $\cup_{t\in [a,b]}K(t)$ is a compact
subset of $N$, or, equivalently, if 
\[
  K:= \{(p,t): t\in [a,b],\, p\in K(t)\}
\]
is a compact subset of $N\times \RR$.
\end{definition}

If $t\in[a,b]\mapsto K(t)$ is a smooth barrier
and if $x\in \partial K(t)$, we let $\nu_K(x,t)$ be the unit normal to $\partial K(t)$
that points out from $K(t)$, we let $\HH_K(x,t)$ denote the dot product of $\nu_K(x,t)$ 
and the mean curvature vector of $\partial K(t)$ at $x$,
 and we let $\vv_K(x,t)$ denote the normal velocity of $\tau\mapsto \partial K(\tau)$ at $(x,t)$ in the direction of 
$\nu_K$.
In terms of a function $f$ as in Definition~\ref{barrier-definition}, 
\begin{equation}\label{f-expressions}
\begin{aligned}
\nu_K &= \frac{\nabla f}{|\nabla f|}, \\
\HH_K &=  -\Div\left( \frac{\nabla f}{|\nabla f|} \right),  \\
\vv_K &=  -\frac{1}{|\nabla f|} \pdf{f}t.
\end{aligned}
\end{equation}
Alternatively, we can describe $\vv_K$ as follows.  Let $I\subset \RR$ be an interval containing $t$
and $\gamma: I\to N$ be a smooth map such that $\gamma(t)=x$ and such that $\gamma(\tau)\in \partial K(\tau)$
for all $\tau\in I$.  Then 
\[
  \vv_K(x,t) = \gamma'(t)\cdot \nu_K(x,t).
\]
For $x\in \partial K(t)$, we define $\Phi_K=\Phi_K(x,t)$ by
\[
   \Phi_K = \vv_K - \HH_K.
\]
Thus $\Phi_K\le 0$ everywhere if and only if $t\mapsto \partial K(t)$ is a subsolution of mean
curvature flow,
and $\Phi_K\ge 0$ if and only it is a supersolution.

For example, let $\lambda>0$, and for $t<0$, let
\[
   K(t)= \{x\in \RR^{m+1}: |x| \ge (-\lambda t)^{1/2} \}.
\]
Thus $\partial K(t)$ is the sphere of radius $(-\lambda t)^{1/2}$ centered at the origin.
At a point $x\in \partial K(t)$, 
$\vv_K(x,t) = \frac12 \lambda^{1/2} |t|^{-1/2}$  and $\HH_K(x,t) = m(\lambda |t|)^{-1/2}$.
Consequently, $\Phi_K$ is positive, zero, or negative according to whether $\lambda$
is greater than, less than, or equal to $(2m)^{1/2}$.

\begin{definition}\label{weak-set-flow-definition}
Let $Z$ be a closed subset of $N\times [T_0,\infty)$,  and for each $t\in [T_0,\infty)$, let 
\[
Z(t):=\{x\in N : (x,t)\in Z\}.
\]
We say that $Z$ is a {\bf weak set flow} (for mean curvature flow)
 with {\bf starting time $T_0$} 
provided the following holds:
if 
\[
  t\in [a,b]\mapsto K(t)
\]
is a smooth compact barrier with $a\ge T_0$, if $K(t)$ is disjoint from $Z(t)$ for all $t\in [a,b)$, and if
$p$ is in the intersection of $K(b)$ and $Z(b)$, then $p\in \partial K(b)$ and 
\[
  \Phi_K(p,b) \ge 0.
\]
\end{definition}
If the starting time is not specified, we take it to be $0$.

For example, consider a smooth barrier $t\in [0,T]\mapsto K(t)$.  
Then $K$ is a weak set flow if and only if $\Phi_K\le 0$ at every $(p,t)$ with $p\in \partial K(t)$.
Similarly, consider a smooth one-parameter family $t\in [0,T]\mapsto M(t)$
of smooth, properly embedded hypersurfaces in $N$.   Then $t\mapsto M(t)$
is a weak set flow if and only it is a classical mean curvature flow.
(These facts follow easily from the definition of weak set flow.)

Note that if $Z$ is a weak set flow and $a\in \RR$, then
\[
   \widetilde{Z}(t) :=  
   \begin{cases}
   Z(t) &\text{if $t\le a$}, \\
   \emptyset &\text{if $t>a$}
   \end{cases}
\]
is also a weak set flow.  Thus weak set flows are allowed to suddenly vanish at any time.

Definition~\ref{weak-set-flow-definition} differs from Ilmanen's original definition, 
but we will show that the two definitions are equivalent
in Section~\ref{equivalent-definitions-section}.

\section{Elementary properties of weak set flows}\label{elementary-section}

\begin{theorem}\label{shrinking-ball-theorem}
Let $m=\dim N -1$ and $c>2m$.  Given $p\in N$, there exists an $\eps>0$
with the following property. 
\begin{enumerate}[\upshape (i)]
\item\label{ball-one} If $0<\delta\le \eps$ and if $0<\tau< \delta^2/c$, then
\begin{equation*}
  t\in [0,\tau] \mapsto K(t):=\{x: \dist(x,p)\le (\delta^2-ct)^{1/2} \}
\end{equation*}
is a smooth compact barrier, and $\Phi_K(x,t)<0$ for all $t\in [0,\tau]$ and $x\in \partial K(t)$.
\item\label{ball-two} If $Z:[T_0,\infty)\mapsto Z(t)$ is a weak set
flow in $N$, then
\[
    f(t)^2 + ct
\]
is a non-decreasing function of $t\in [T_0,\infty)$, where
\[
f(t) = \min\{\eps, \dist(Z(t),p)\}.
\]
\end{enumerate}
\end{theorem}

\begin{proof}
For $r>0$, 
let $B(r)=\{x: \dist(x,p)\le r\}$.
Choose $\eps>0$ so that for $r\in (0,\eps]$, the geodesic sphere $\partial B(r)$
is smooth and compact, and
\begin{equation}\label{small-ball-bound}
   \HH(B(r)) > -\frac{c}{2r}.
\end{equation}
(This is possible since $\HH(B(r)) = -(m/r)+o(r)$.)
Assertion~\eqref{ball-one} follows immediately.

Suppose Assertion~\eqref{ball-two} is false.  Then there exist $T< T+\tau$ such that 
$f(T)^2+cT$ is greater than $f(T+\tau)^2+c(T+\tau)$. That is,
\[
   f(T)^2   > f(T+\tau)^2 + c\tau.
\]
By relabeling, it suffices to consider the case $T=0$: 
\[
   f(0)^2 > f(\tau)^2 + c\tau.
\]
Thus $f(\tau)<f(0)\le \eps$, so $f(\tau)=\dist(Z(t),p)$.  
Choose $\delta$ with 
with $f(0)> \delta > (f(\tau)^2+c\tau)^{1/2}$.  Thus
\begin{equation}\label{delta-squeeze}
   \min\{\eps, \dist(Z(0),p)\} > \delta > \left(\dist(Z(\tau),p)^2+c\tau\right)^{1/2}.
\end{equation}
Define $K(\cdot)$ by
\[
   t \in [0, \tau] \mapsto K(t):= B((\delta^2-ct)^{1/2}).
\]
Note by~\eqref{delta-squeeze} that the radius of the ball $K(t)$ is strictly between $0$ and $\delta<\eps$
for all $t\in [0,\tau]$.  Thus $K$ is a smooth compact barrier.
By Assertion~\eqref{ball-one}, $\Phi_K<0$ at all points of $\partial K(\cdot)$.
On the other hand, from~\eqref{delta-squeeze} we see that $K(t)$ and $Z(t)$ are disjoint
at time $0$ but not at time $\tau$, a contradiction.
\end{proof}

\begin{corollary}\label{shrinking-ball-corollary}
Suppose $T>T_0$.
\begin{enumerate}[\upshape (i)]
\item\label{ball-corollary-one} If $p\in Z(T)$, then $\dist(Z(t),p)^2 \le c(T-t)$ for $t<T$ close to $T$.
\item\label{ball-corollary-two} If $u:N\times \RR\to\RR$ is continuous, then
\[
     \tilde u(T)\ge \limsup_{t\uparrow T}\tilde u(t),
\]
where $\tilde u(t):=\inf_{x\in Z(t)} u(x,t)$.
\end{enumerate}
\end{corollary}

\begin{proof}
In the notation of Theorem~\ref{shrinking-ball-theorem}, $f(T)=0$,
so $f(t)^2+ct \le CT$ for $t\le T$, and therefore
\[
   \min\{\eps, \dist(Z(t),p)\}^2= f(t)^2 \le c(T-t),
\]
which proves Assertion~\eqref{ball-corollary-one}.

To prove Assertion~\eqref{ball-corollary-two}, let $p\in Z(T)$.
By Assertion~\eqref{ball-corollary-one}, if $t<T$ is sufficiently close to $T$, 
then there exists a point $p(t)\in Z(t)$
closest to $p$, and $p(t)\to p$ as $t\to T$. Now $\tilde u(t)\le u(p(t),t)$, so
\[
\limsup_{t\uparrow T} \tilde u(t)
\le
\limsup_{t\uparrow T}u(p(t),t) 
=
u(p,T).
\]
Assertion~\eqref{ball-corollary-two} follows by taking the infimum over all $p\in Z(T)$.
\end{proof}

\begin{theorem}\label{finite-speed-theorem}
For every $r>0$, $\lambda\in \RR$, and positive integer $n$, 
there is a constant $h=h(r,\lambda,n)>0$  with the following 
property.   Suppose that $N$ is a smooth Riemannian $n$-manifold,
that $R>r$, that the geodesic ball $\overline{B}(p,R)$ in $N$ is compact, 
and that the Ricci curvature of $N$
is $\ge \lambda$ on $\overline{B}(p,R)$.
If $t\in [0,T]\mapsto Z(t)$ is a weak set flow in $N$ and if
   $\dist(Z(0),p) > R$, then 
\begin{equation}\label{swango}
   \dist(Z(t),p) > R - ht  \tag{*}
\end{equation}
for all $t\in [0,T]$ with $t\le  (R-r)/h$.
\end{theorem}

\begin{proof}
Let $\mathcal{H}$ be a complete $n$-dimensional manifold that has the same dimension as $N$,
that has constant sectional curvature, and that has Ricci curvature equal to the minimum of $0$ and $\lambda$.  
Let $h>0$ be the mean curvature of a sphere of
radius $r/2$ in $\mathcal{H}$.

Suppose, contrary to the theorem, that~\eqref{swango} fails for some time  $t\le (R-r)/h$.
Let $\tau$ be the first such time.  
By Corollary~\ref{shrinking-ball-corollary} (applied to $\tilde u(t):=\dist(p,Z(t))$),
\[
  \dist(Z(\tau),p) = R- h\tau.
\]
Since $0 < \tau\le (R-r)/h$, 
\[
  r\le  \dist(Z(\tau),p)  < R.
\]
Let $q$ be a point in $Z(\tau)$ with $\dist(p,q)=\dist(p,Z(\tau))=R-m\tau$.
Let  $\gamma$ be a unit-speed, shortest geodesic from $p$ to $q$, prolonged to be a geodesic of length $R$:
\begin{align*}
&\gamma:[0, R]\to N, \\
&\gamma(0)=p, \\
&\gamma(R-h\tau)=q.
\end{align*}
Let
\[
 K: t\in [0,\tau] \mapsto K(t):=\overline{B}(\gamma(R-ht-r/2),r/2).
\]
Since $\gamma$ is length minimizing on $[0,R-h\tau]$, it follows that
the function 
\[
  (x,y)\in N\times N\mapsto \dist(x,y)
\]
is smooth in a small neighborhood of $(x,y)$ if $x$ and $y$ are points in $\gamma((0,R-h\tau))$.

Thus $K$ is smooth in a spacetime neighborhood of $(q,\tau)$.

Note that $K(t)$ is disjoint from $Z(t)$ for $t<\tau$ and  that $K(\tau)\cap Z(\tau)=\{q\}$.

Thus
\[
    \Phi_K(q,\tau):= \vv_K(q,\tau) - \HH_K(q,\tau) \ge 0,
\]
so $\vv_K(q,\tau)\ge \HH_K(q,\tau)$.\footnote{To conclude above that $\Phi_K(q,\tau)\ge 0$ using the definition of weak set flow, the barrier $K$
should be smooth everywhere.  However, it suffices for $K$
to be smooth in a spacetime neighborhood of $(q,\tau)$. 
See Theorem~\ref{nonsmooth-barriers-theorem}.}
However, $\vv_q= -h$, and $\HH_K(q,\tau)>-h$ 
by mean curvature comparison (see~\cite[Lemma 7.1.2]{Petersen}
or~\cite{Dai-Wei}*{Theorem~1.2.2}).  Thus $\vv_K(q,\tau) < \HH_K(q,\tau)$, a contradiction.
\end{proof}

\begin{theorem}[Ilmanen~\cite{ilmanen-generalized}*{Theorem~6.4}]\label{compact-theorem}
Suppose that $Z:[0,\infty)\mapsto Z(t)$ is a weak set flow in a complete Riemannian
$n$-manifold with Ricci curvature bounded below by $\lambda$.
Then
\[
    Y(t) \subset  \{x: \dist(x, Y(0)) \le r + ht\}
\]
for all $t>0$.   
Here $r$ can be any positive number, and $h=h(r,\lambda,n)$ is as in Theorem~\ref{finite-speed-theorem}.

In particular, if $Z(0)$ is empty, then $Z(t)$ is empty for all $t$,
and if $Z(0)$ is compact, then $\cup_{t\le T}Z(t)$ is compact for  $T<\infty$.
\end{theorem}

\begin{proof} 
If $\dist(p,Y(0)) > r + ht$, then $\dist(p,Y(t)) \ge r$, and so $p\notin Y(t)$.
Thus if $p\in Y(t)$, then $\dist(p,Y(0))\le r+ht$.
\end{proof}

The ``in particular" assertions of Theorem~\ref{compact-theorem} are
 false (in general) without the lower bound on Ricci curvature.
For example, let $\Sigma$ be a compact manifold with Riemannian metric $\sigma$, 
and let $N=\RR\times \Sigma$ with the complete metric $dx^2 + (\exp(-x-x^3/3))^2\sigma$.
Then $t\mapsto M(t):=\{\tan t\}\times \Sigma$ is a mean curvature flow with $M(0)$
compact and $\cup_{t\in [0,\pi/2]}M(t)$ noncompact, and $t\mapsto M(t-\pi/2)$ is a 
mean curvature flow with $M(t)$ empty for $t=0$ but nonempty for $t\in (0,\pi)$.

\section{Barrier Modification}\label{barrier-modification-section}

\begin{lemma}\label{first-modification}
Suppose that $U$ is an open subset of $N$, that $t\in [a,b]\mapsto K(t)$
is a smooth barrier in $U$, and that $p\in U\in \partial K(b)$.
Then there is $\hat{a}\in [a,b)$ and a smooth compact barrier
\[
  t\in [\hat{a},b] \mapsto \Kup(t)
\]
in $U$ such that
\begin{gather*}
 \Kup(t) \subset \interior{K(t)} \quad \text{for $t\in [\hat{a},b)$}, \\
\Kup(b)\cap \partial K(b)=\{p\}, \,\text{and} \\
  \Phi_{\Kup}(p,b)=\Phi_K(p,b).
\end{gather*}
\end{lemma}

\begin{proof}
Let $f:U\times [a,b]\to \RR$ be as in Definition~\ref{barrier-definition}.
By postcomposing with a smooth bounded function, we can assume that $f$ is bounded.
Let $\phi:U\times [a,b]\to \RR$ be a smooth proper function such that $\phi$ vanishes to infinite
order at $(p,b)$ and such that $\phi>0$ at all other points.  By Sard's Theorem, almost every $c$
is a non-critical value of 
\[
   (x,t) \in (U\times[a,b])\setminus (p,b)  \mapsto -\frac{f(x,b)}{\phi(x,b)}.
\]
Choose such a $c>0$, and let
\[
  \fup(q,t):= f(q,t) + c\phi(q,t).
\]
Since $f$ is bounded and $\phi$ is proper,
\[
   \{(x,t)\in U\times [a,b]: \fup(x,t)\le 0\}
\]
is compact.
By choice of $c$, $\nabla\fup$ does not vanish anywhere on $\{x\in U: \fup(x,b)=0\}$.
Now choose $\hat{a}\in [a,b)$ sufficiently close to $a$ that $\nabla\fup$ does not
vanish anywhere on $\{(x,t): t\in [\hat{a},b], \, \fup(x,t)=0\}$.
\end{proof}

\begin{theorem}[Noncompact, nonsmooth barriers]\label{nonsmooth-barriers-theorem}
Suppose that $f:N\times [a,b]\to \RR$ is continuous, and let $K(t)=\{x:f(x,t)\le 0\}$ for $t\in [a,b]$.
Suppose that $Z$ is a weak set flow in $N$ with starting time $T_0<b$, that
$Z(t)$ is disjoint from the interior of $K(t)$ for all $t<b$, and that $p\in Z(b) \cap \partial K(b)$,

If $f$ is smooth in a spacetime neighborhood of $(p,b)$ and if $\nabla f(p,b)$ is nonzero,
then $\Phi_K(p,b)\ge 0$.
\end{theorem}

\begin{proof}
Choose $U$ and $\eps$ small enough that $t\in [b-\eps,b]\mapsto K(t)\cap U$ is a smooth barrier in $U$.
By Lemma~\ref{first-modification}, there is smooth compact barrier
 $t\in [\hat{a},b]\mapsto \Kup(t)\subset U\cap K(t)$ such that $p\in \partial \Kup(b)$ and
    $\Phi_{\Kup}(p,b)=\Phi_K(p,b)$. 
By definition of weak set flow, $\Phi_{\Kup}(p,b)\ge 0$.
 \end{proof}

\begin{theorem}[Barrier Modification Theorem]\label{barrier-modification-theorem}
Suppose that $t\in [a,b]\mapsto K(t)$ is a smooth compact barrier in $N$,
that $p\in \partial K(b)$, and that 
\[
   \Phi_K(p,b):=\vv_K(p,b)-\HH_K(p,b)<\eta.
\]
Then there is an $\hat{a} \in [a,b)$ and a smooth compact
barrier $t\in [\hat{a},b]\mapsto \Kup(t)$ with the following properties:
\begin{enumerate}[\upshape (1)]
\item $\Kup(t)$ is contained in $K(t)$ for all $t\in [\hat{a},b]$.
\item $p\in \partial \Kup(b)$.
\item\label{lemma-claim-4} 
\[
   \lim_{x\in \partial K(b), \, x\to p} \frac{\dist(x, \Kup(b))}{\dist(x,p)^2}  > 0.
\]
\item\label{lemma-claim-5} 
    $\Phi_{\Kup}(x,t):=\vv_{\Kup}(x,t) - \HH_{\Kup}(x,t) < \eta$ 
    for all $(x,t)$ with $t\in [\hat a,b]$ and $x\in \partial \Kup(t)$.
\end{enumerate}
\end{theorem}

\begin{proof} 
It suffices to consider the case $[a,b]=[a,0]$.
Let $f:N\times [a,0]\to \RR$ be as in Definition~\ref{barrier-definition}.  By multiplying $f$ by a constant,
we can assume that $|\nabla f(p,0)|=1$.  We can also assume that $f$ is proper, i.e., that $f(p_i)\to\infty$
provided $p_i$ is a divergent sequence in $N$.

Let $\delta(\cdot)$ be a smooth bounded function on $N$ that is positive on $N\setminus\{p\}$ and that
coincides with $\frac12\dist(\cdot,p)^2$
in a neighborhood of $p$. 
Let 
\begin{align*}
&\tf: N\times [a,0]\to \RR, \\
&\tf(x,t) = f(x,t) + c \left(\delta(x) - t \right),
\end{align*}
and let
\[
   \tK(t) = \{(x,t): \tf \le 0\},
\]
where $c$ is a positive constant that will be specified below.  

Since $0$ is a regular value of $f(\cdot,0)$, there is an $\eps>0$ 
such that $0$ is a regular value of $\tf(\cdot,0)$ provided $c\in [0,\eps]$.
Fix a $c\in (0,\eps]$ such that $\Phi_{\tK}(p,0)<\eta$.
(This is possible since $\Phi_{\tK}(p,0)$ depends continuously on $c$.)

Since $\tf\ge f$ with strict inequality except
at $(p,0)$, we see that 
\begin{equation}\label{tilde-K-inside}
   \tK(t) \subset\interior(K(t))  \quad\text{for $t\in [a,0)$}
\end{equation}
and
\[
  \tK(0) \setminus \interior(K(0)) = \{p\}.
\]

Note also that
\[
  \lim_{x\in \partial \tK(0), \, \dist(x,p)\to 0} \frac{\dist(x,\partial K(0))}{\dist(x,p)^2} = c.
\]

Let $\psi:N\to \RR$ be a smooth, bounded, nonnegative function such that $\psi$ vanishes
on an open set $U$ containing $p$ and such that $\psi>0$ at all points of the set
\[
  \Sigma:= \{x\in \partial \tK(0): \Phi_{\tK}(x,0) \ge \eta \}.
\]

Now let
\begin{align*}
&\fup: N\times [a,0]\to\RR, \\
&\fup(x,t) = \tf(x,t) + t\Lambda \psi(x),
\end{align*}
and
\[
  \Kup(t)= \{x:  \fup(x,t)\le 0\} \qquad (t\in [a,0]),
\]
where $\Lambda$ is a positive constant that will be specified below.

Note that $\fup(\cdot,0)=\tf(\cdot,0)$, so $\Kup(0)=\tK(0)$.
Thus for $x\in \partial \Kup(0)$,
\[
    \HH_{\Kup}(x,0)= \HH_{\tK}(x,0)
\]
and 
\begin{align*}
\vv_{\Kup}(x,0) 
&=  \vv_{\tK}(x,0) - \Lambda w(x)
\end{align*}
(by~\eqref{f-expressions}), where 
\[
   w(x) = \frac{\psi(x)}{|\nabla \fup(x,0)|} = \frac{\psi(x)}{|\nabla \tf(x,0)|}.
\]
Consequently,
\[ 
\Phi_{\Kup}(x,0) = \Phi_{\tK}(x,0) -  \Lambda w(x).
\]
Thus if $x\notin\Sigma$, then
\begin{equation*}
 \Phi_{\Kup}(x,0) \le  \Phi_{\tK}(x,0) < \eta, 
\end{equation*}
and if $x\in \Sigma$, then
\begin{equation*}
\Phi_{\Kup}(x,0)
<
(\max_{y\in\Sigma} \Phi_{\tK}(y,0)) - \Lambda (\min_\Sigma w). 
\end{equation*}
Choose $\Lambda>0$ large enough that this last expression is $<\eta$.  Then
\begin{equation*}
\Phi_{\Kup}(x,0) < \eta  \qquad (x\in \partial \Kup(0)).
\end{equation*}

Since
\[
   \Kup(0)\setminus U \subset \interior(K(0)),
\]
there is an $\hat{a}\in [a,0)$ such that
\[
   \Kup(t) \setminus U \subset \interior(K(t)) \quad\text{for all $t\in [\hat{a},0]$.}
\]  
On the other hand, since $\psi$ vanishes on $U$,
\[
    \Kup(t)\cap U = \tK(t) \cap U \subset \interior(K(t)) \quad\text{for all $t\in [\hat{a},0)$}
\]
by~\eqref{tilde-K-inside}. 
Thus
\[
   \Kup(t)\subset \interior{K(t)} \quad\text{for all $t\in [\hat{a},0)$}.
\]

Finally, since $0$ is a regular value of $\fup(\cdot,0)$ and since $\Phi_{\Kup}(\cdot,0)<\eta$ everywhere 
on $\partial \Kup(0)$,
we can choose $\hat{a}$ close enough to $0$ to guarantee that $0$ is a regular value of $\fup(\cdot,t)$ for all 
 $t\in [\hat{a},0]$ and that $\Phi_K(\cdot,t)<\eta$ everywhere on $\partial \Kup(t)$ for all $t\in [\hat{a},0]$.
 \end{proof}

\section{Bounds on the Distance Function}\label{distance-function-section}

\begin{lemma}\label{key-lemma}
Suppose that $N$ is a smooth Riemannian manifold, 
that $Z(\cdot)$ is a weak set flow in $N$ with starting time $T_0$,
 and that
\[
   t \in [a,b] \mapsto K(t)
\]
is a smooth barrier with $a\ge T_0$.
 Suppose that $\lambda\in \RR$, that
\[
   t\in [a,b] \mapsto e^{-\lambda t}\dist(K(t),Z(t))
\]
attains a positive minimum at time $t=b$,
and that there is a geodesic 
\[
   \gamma:[0,L] \to N
\]
parametrized by arclength such that
\begin{align*}
p &:=\gamma(0)\in K(b), \\
q &:= \gamma(L)\in Z(b), \quad\text{and} \\
L &= \dist(K(b),Z(b)).
\end{align*}
If $\Ricci(\gamma',\gamma')>\lambda$ on $[0,L]$, 
then 
\[
  \Phi_K(p,b) > 0.
\]
\end{lemma}

\begin{proof}
We may assume that $b=0$. 
Unfortunately, the signed distance function to $\partial K(0)$ need not
be smooth at the point $q=\gamma(L)$.  (It is smooth in a neighborhood
of each $\gamma(s)$ with $s\in [0,L)$.)   
 We will use the Barrier Modification Theorem~\ref{barrier-modification-theorem} 
 to get around
the lack of smoothness.

Let $\lambda_0$ be the minimum of $\mathrm{Ric}(\gamma',\gamma')$ on $[0,L]$.
Thus $\lambda_0>\lambda$.
We will prove the lemma by proving that 
\begin{equation}\label{sharper}
    \Phi_K(p,0)\ge (\lambda_0 - \lambda) L.
\end{equation}
Suppose that~\eqref{sharper} does not hold.
Then by the Barrier Modification Theorem~\ref{barrier-modification-theorem}, there is a smooth, compact
barrier
\[
   t\in [-\eps,0]\mapsto \Kup(t)
\]
such that:
\begin{equation}\label{all-stuff}
\begin{gathered}
\Kup(t) \subset K(t) \quad (t\in [-\eps,0]), \\
\Kup(t)\subset \operatorname{\interior}(K(t)) \quad (t<0), \\
\Kup(0)\cap \partial K(0)= \{p\}, \\
\liminf_{x\in \partial K(0), \, x\to p} \frac{\dist(x, \Kup(0))}{\dist(x,p)^2} > 0, \\
\Phi_{\Kup}(p,0) < (\lambda_0-\lambda)L.
\end{gathered}
\end{equation}

Note that $\dist(\cdot,\Kup(0))$ is smooth
on an open set containing $q$, and thus 
\[
   (t,x) \mapsto \dist(x, e^{\lambda t} \Kup(t))
\]
is smooth on an open spacetime set containing $(q,0)$.  
For $t\in [-\eps,0]$, let
\[
   \tK(t) =  \{x\in W: e^{-\lambda t} \dist(x,\partial \Kup(t)) \le L \}.
\]
By~\eqref{all-stuff},
    $Z(t)\cap W$ and $\widetilde{K}(t)$ are disjoint for $t<0$
    and $Z(0)\cap \widetilde{K}(0)=\{q\}$.
Thus
\begin{equation}\label{rosebud}
   \Phi_{\widetilde{K}}(q,0) \ge 0
\end{equation}
by Theorem~\ref{nonsmooth-barriers-theorem}.

A standard computation
 (cf.~\cite{White_controlling}*{Lemma~12.2}) shows that
\begin{equation}\label{sagan}
\begin{aligned}
  \HH_{\widetilde{K}}(q,0) 
  &\ge  \HH_{\Kup}(p,0) + \int_0^L \Ricci_{\gamma(s)}(\gamma'(s),\gamma'(s)) \,ds \\
  &\ge  \HH_{\Kup}(p,0) + \lambda_0 L.
\end{aligned}
\end{equation}
Note also that
\begin{equation}\label{sophie}
  \vv_{\widetilde{K}}(q,0) 
  = \vv_{\Kup}(p,0) + \lambda L. 
\end{equation}

By~\eqref{rosebud}, \eqref{sagan}, and~\eqref{sophie},
\begin{align*}
0 
&\le \Phi_{\widetilde{K}}(q,0) \\
&= \vv_{\widetilde{K}}(q,0) - \HH_{\widetilde{K}}(q,0) \\
&\le \vv_{\Kup}(p,0) - \HH_{\Kup}(p.0) + (\lambda - \lambda_0)L \\
&= \Phi_{\Kup}(p,0) + (\lambda-\lambda_0)L,
\end{align*}
contradicting~\eqref{all-stuff}.
\end{proof}

\begin{proposition}\label{key-proposition}
Suppose $t\in [0,T]\mapsto K(t)$ is a compact smooth barrier such that $\Phi_K(p,t)\le 0$ 
for all $t\in [0,T]$ and $p\in \partial K(t)$.
Suppose that $\eta>0$, that $\lambda\in \RR$, that the set
\[
     Q:= \{ (p,t): t\in [0,T], \, \dist(p,K(t)) \le e^{\lambda t}\eta \}
\]
is compact, and that $\lambda$ is a strict lower bound for Ricci curvature on $\cup_{t\in [0,T]}Q(t)$.
If $t\in [0,T]\mapsto Z(t)$ is a weak set flow and if
\[
   \dist(Z(0), K(0)) > \eta,
\]
then
\[
  \dist(Z(t),K(t)) > e^{\lambda t} \eta
\]
for all $t\in [0,T]$.
\end{proposition}


\begin{proof}
Suppose not.  Then $Q\cap Z$ is a nonempty compact subset of space time.
Thus there is first time $t$ such that $Q(t)\cap Z(t)$ is nonempty. By hypothesis, $t>0$.
Let $q\in Z(t)\cap Q(t)$:
\[
   \dist(q, K(t)) \le e^{\lambda t} \eta.
\]
Thus $\dist(q,K(t))=e^{\lambda t}\eta$ by
 Corollary~\ref{shrinking-ball-corollary} (applied to the function $u(x,\tau)=\dist(x, K(\tau))$).
Let $p\in K(t)$ be a point closest to $q$.
By Lemma~\ref{key-lemma}, $\Phi_K(p,t)>0$, a contradiction.
\end{proof}

\begin{corollary}\label{key-corollary}
In Proposition~\ref{key-proposition}, if $t\mapsto Z(t)$ is a weak set flow in $N$ such that $\dist(Z(0),K(0))\ge \eta$,
then $\dist(Z(t),K(t))\ge e^{\lambda t} \eta$ for all $t\in [0,T]$.
\end{corollary}

\begin{proof}
By Proposition~\ref{key-proposition}, 
\[
   \dist(Y(t),V(t)) > e^{\lambda t}\delta   \quad (t\in [0,T])
\]
holds for all $\delta$ with $0<\delta<\eta$.
The result follows immediately.
\end{proof}

\begin{remark}\label{key-remark}
Proposition~\ref{key-proposition} and Corollary~\ref{key-corollary}
 remain true if $t\in [0,T]\mapsto K(t)$ is a smooth mean curvature flow of closed hypersurfaces.
 The proofs are the same except for minor changes of notation.
Furthermore, in this case, one sees from the proof that it is not necessary for the flow to be smooth at
the initial time:
$K$ can be any compact subset of $N\times [0,T]$ such that $t\in (0,T]\mapsto K(t)$
is a smooth mean curvature flow.
\end{remark}

\begin{theorem}\label{distance-theorem}
Suppose $t\in [0,T]\mapsto Y(t)$ and $t\in [0,T]\mapsto Z(t)$ 
are weak set flows in a smooth Riemannian manifold $N$.
Suppose $\eta>0$ and $\lambda$ are such that the set
\[
   Y^\eta_\lambda := \cup_{t\in [0,T]} \{ p: \dist(p,Y(t)) \le e^{\lambda t}\eta\}
\]
is compact.   Suppose also that $\lambda$ is a lower bound for Ricci curvature of $N$
on the set $Y^\eta_\lambda$.
If 
\begin{equation}\label{key-inequality}
   \dist(Z(t),Y(t)) \ge e^{\lambda t} \eta
\end{equation}
holds for $t=0$, then it holds for all $t\in [0,T]$.
\end{theorem}

\begin{proof}
{\bf Case 1}: $\lambda$ is a strict lower bound for Ricci curvature on $Y^\eta_\lambda$.
Let $\mathcal{T}$ be the set of times $\tau\in [0,T]$ such that~\eqref{key-inequality} holds for all $t\in [0,\tau]$.
By hypothesis, $0\in \mathcal{T}$.  
Thus $\mathcal{T}$ is either $[0,b]$ or $[0,b)$ where $b=\sup\mathcal{T}$.
Let $y\in Y(b)$ and $z\in Z(b)$.  Then for $t<b$,
\begin{align*}
e^{\lambda t}\eta
&< \dist(Y(t),Z(t))   \\
&\le \dist(Y(t),y) + \dist(y,z) + \dist(z, Z(t)).
\end{align*}
Taking the limit as $t\uparrow b$ gives (see Corollary~\ref{shrinking-ball-corollary})
\[
e^{\lambda b}\eta \le \dist(y,z).
\]
Taking the infimum over $y\in Y(b)$ and $z\in Z(b)$ gives $e^{\lambda b}\eta \le \dist(Y(b),Z(b))$.
Thus $\mathcal{T}=[0,b]$.

Hence it suffices to show that if $\tau<T$ is in $\mathcal{T}$, then $\tau+\eps\in \mathcal{T}$ for some $\eps>0$.
Consider such a time $\tau$.  Let
\[
   J=\{p\in N: \dist(p, Y(\tau)) \ge e^{\lambda \tau}\eta\}.
\]
Since $\tau\in \mathcal{T}$, we see that $Z(\tau)\subset J$.
By Theorem \ref{app_thm}, there exists a closed $C^1$ hypersurface $M$ in $N$ such that $M$ separates $Y(\tau)$ and $J$
and such that
\[
     \dist(Y(\tau),M) = \dist(M, J) = \frac12\dist(Y(\tau),J) = \frac12 e^{\lambda\tau}\eta.
\]
Existence of such a hypersurface that is $C^{1,1}$ was sketched in~\cite{Ilm_mani_surv}*{Lemma~4G} and proved in~\cite{Bernard_patrick}. (See also~\cite{Fathi_insertion}.) 
Since existence of such an $M$ that is merely $C^1$ 
suffices for our application and is simpler to prove, we provide an existence proof in the appendix.

By the Local Regularity Theorem~\cite{White_reg}, there exists a smooth mean curvature flow 
\[
  t\in (\tau, \tau+\eps] \mapsto M(t)
\]
such that $M(t)$ converges in $C^1$ to $M$ as $t\to\tau$.  Accordingly, we set $M(\tau)=M$.

Let 
\[
\tM  := \cup_{t\in [\tau,\tau+\eps]}\{p: \dist(p,M(t)) \le \frac12e^t\eta  \}.
\]
By replacing $\eps$ by a smaller $\eps>0$, we can assume that $[\tau,\tau+\eps]\subset [0,T]$,
that $\tM$ is compact, and that $\lambda$ is a strict lower bound for Ricci curvature of $N$
on $\tM$.

Consequently, 
\begin{align*}
 \dist(Y(t),M(t)) 
 &\ge e^{\lambda (t-\tau)}\dist(Y(\tau),M(\tau)) \\
 &= e^{\lambda (t-\tau)}\frac12 e^{\lambda \tau} \eta  \\
 &= \frac12 e^{\lambda t}\eta
\end{align*}
and
\begin{align*}
\dist(Z(t),M(t)) 
&\ge e^{\lambda (t-\tau)}\dist(Z(\tau),M(\tau)) \\
&\ge e^{\lambda (t-\tau)}\dist(J,M) \\
&= e^{\lambda(t-\tau)} \frac12 e^{\lambda \tau}\eta \\
&= \frac12 e^{\lambda t}\eta.
\end{align*}
for $t\in [\tau,\tau+\eps]$ by Corollary~\ref{key-corollary} and Remark~\ref{key-remark}, with the time
interval $[\tau,\tau+\eps]$ in place of $[0,T]$.
Since $M(t)$ separates $Y(t)$ and $Z(t)$,
\begin{align*}
\dist(Y(t),Z(t))
&\ge
\dist(Y(t),M(t)) + \dist(M(t),Z(t)) \\
&\ge
e^{\lambda t}\eta
\end{align*}
for $t\in [\tau,\tau+\eps]$.  
Thus $[\tau,\tau+\eps]\subset \mathcal{T}$.  This proves the theorem in Case 1.

{\bf Case 2}: $\lambda$ is any lower bound for Ricci curvature on $Y^\eta_\lambda$.
Taking any $\lambda'<\lambda$, then $Y^\eta_{\lambda'}\subset Y^\eta_\lambda$, and thus $\lambda'$
is a strict lower bound for Ricci curvature on $Y^\eta_{\lambda'}$.  Thus by Case 1,
\[
     \dist(Y(t),Z(t)) \ge e^{\lambda' t} \eta 
\]
for all $t\in [0,T]$.   Since this inequality holds for every $\lambda'<\lambda$, it also holds for $\lambda'=\lambda$.
\end{proof}

\begin{theorem}\label{long-time-theorem}
Suppose $N$ is a complete Riemannian manifold with Ricci curvature bounded below by $\lambda$.
Suppose $t\in [0,\infty)\mapsto Y(t)$ and $t\in [0,\infty)\mapsto Z(t)$ are weak set flows with 
$Y(0)$ compact.  Then
\[
   t\in [0,T] \mapsto e^{-\lambda t} \dist(Y(t),Z(t))
\]
is non-decreasing.
\end{theorem}

\begin{proof}
Let $0<T<\infty$.
By Theorem~\ref{compact-theorem}, $\cup_{t\le T}Y(t)$ is a compact subset of $N\times [0,T]$.
By Theorem~\ref{distance-theorem}, 
\begin{equation}\label{starbucks}
     \dist(Y(t),Z(t)) \ge e^{\lambda t} \dist(Y(0),Z(0)).
\end{equation}
for $t\le T$.  
(Note that the hypotheses of Theorem~\ref{distance-theorem} are satisfied for every $\eta>0$.)
Since $T$ is arbitrary,~\eqref{starbucks} holds for all $t\ge 0$.  
The same argument shows that 
\[
    \dist(Y(\tau+t),Z(\tau+t)) \ge e^{\lambda t}\dist(Y(\tau),Z(\tau))
\]
for $0\le \tau < t$.  
\end{proof}


\section{The Avoidance Theorem}\label{avoidance-section}

\begin{theorem}\label{avoidance-theorem}
Let $t\in [0,T]\mapsto Y(t)$ and $t\in [0,T]\mapsto Z(t)$ be weak set flows in $N$
such that $C:=\cup_{t\in [0,T]}Y(t)$ is compact.
If $Y(t)$ and $Z(t)$ are disjoint at time $0$, then they are disjoint at every time $t\in [0,T]$.
\end{theorem}

\begin{proof}
Let $U$ be an open set containing $C$ with $\overline{U}$ compact.
Let $\lambda$ be a lower bound for Ricci curvature on $U$, and choose $\eta$
with $0<\eta<\dist(Y(0),Z(0))$ 
 sufficiently small that
\[
    \cup_{t\in [0,T]}\{x: \dist(x, Y(t)) \le e^{\lambda t}\eta\}
\]
lies in $U$.  Then $\dist(Y(t),Z(t)) \ge e^{\lambda t}\eta$ for all $t\in [0,T]$
by Theorem~\ref{distance-theorem}.
\end{proof}


\section{Equivalent Definitions of Weak Set Flow}\label{equivalent-definitions-section}

A {\bf strong mean-curvature-flow barrier}, or {\bf strong barrier} for short, is
a smooth compact barrier $t\in [a,b]\mapsto K(t)$ such that $\Phi_K(x,t)<0$
for all $(x,t)$ with $x\in \partial K(t)$ and $t\in [a,b]$.  (Recall that the barrier $K$ is said to be compact
if $\cup_t K(t)$ is compact, or, equivalently, if $K$ is a compact subset of $N\times\RR$.)

\begin{theorem}\label{equivalent-definitions-theorem}
Let $Z$ be a closed subset of $N\times[T_0,\infty)$.
The following are equivalent:
\begin{enumerate}[\upshape (1)]
\item\label{equivalent-1} $Z$ is a weak set flow (as in Definition~\ref{weak-set-flow-definition}) with starting time $T_0$.
\item\label{equivalent-2} If $t\in [a,b]\mapsto K(t)$ is a strong barrier with $a\ge T_0$ and
if $K(a)$ is disjoint from $Z(a)$, then $K(t)$ is disjoint from $Z(t)$ for all $t\in [a,b]$.
\item\label{equivalent-3} If $t\in [a,b]\mapsto K(t)$ is a strong barrier with $a\ge T_0$ and
if $K(a)$ is disjoint from $Z(a)$, then $K(t)$ is disjoint from $Z(t)$ for all $t\in [a,b)$.
\item\label{equivalent-4}  If $t\in [a,b]\mapsto M(t)$ is a smooth mean curvature flow
 of closed, embedded, hypersurfaces
with $a\ge T_0$, and if $M(0)$ is disjoint from $Z(0)$, then $M(t)$ is disjoint from $Z(t)$ for all $t\in [a,b]$.
\end{enumerate}
\end{theorem}

\begin{proof}
Trivially~\eqref{equivalent-1} implies~\eqref{equivalent-2}.

To see that~\eqref{equivalent-2} implies \eqref{equivalent-1}, suppose to the contrary that \eqref{equivalent-2} holds but that~\eqref{equivalent-1} fails.  
Then there is a smooth 
compact barrier $[a,b]\mapsto K(t)$ such that $Z(t)\cap K(t)$ is empty for $t\in [a,b)$
and such that $K(t)\cap Z(t)$ contains a point $p$ such that
\begin{equation*}\label{dichotomy}
\begin{aligned}
&\text{$p\in \interior(K(b))$, or} \\
&\text{$p\in \partial K(b)$ and $\Phi_K(b)<0$.}
\end{aligned}
\end{equation*}
By Corollary~\ref{shrinking-ball-corollary}
(applied to the function $u(x,t)=\dist(x, N\setminus K(t))$)\footnote{Corollary~\ref{shrinking-ball-corollary} was proved for weak set flows.
But Corollary~\ref{shrinking-ball-corollary}
is based on Theorem~\ref{shrinking-ball-theorem}, and the only barriers 
in the proof of that result were strong barriers.  Thus Theorem~\ref{shrinking-ball-theorem}
and Corollary~\ref{shrinking-ball-corollary} also hold for flows having Property~\eqref{equivalent-2}.}  
     $p$ must be in $\partial K(b)$, and so $\Phi_K(b)<0$.
By the Barrier Modification Theorem~\ref{barrier-modification-theorem},
 there is a strong barrier $t\in [b-\eps,b]\mapsto \Kup(t)$
such that $\Kup(t)\subset K(t)$ for all $t\in [b-\eps,b]$ and such that $p\in \partial \Kup(b)$.
But that violates~\eqref{equivalent-2}.

Thus we have proved that \eqref{equivalent-1} and~\eqref{equivalent-2} are equivalent.

Trivially~\eqref{equivalent-2} implies~\eqref{equivalent-3}.  The reverse implication holds
because any strong barrier $t\in [a,b]\mapsto K(t)$ can be prolonged to a strong
 barrier on a slightly larger time interval $[a,b+\eps]$.
 
 The Avoidance Theorem~\ref{avoidance-theorem} shows that~\eqref{equivalent-1} implies~\eqref{equivalent-4} 
 (since any smooth mean curvature flow is a weak set flow). 

It remains only to show that~\eqref{equivalent-4} implies~\eqref{equivalent-2},
or, equivalently, that failure of~\eqref{equivalent-2} implies failure of~\eqref{equivalent-4}.
Thus suppose that~\eqref{equivalent-2} does not hold, i.e., that
that there is a strong barrier $t\in [a,b]\mapsto K(t)$ with $a\ge T_0$ such that
$K(a)$ is disjoint from $Z(a)$
 but $K(t)\cap Z(t)$ is nonempty for some time $t\in (a,b]$.  
 By relabeling, we may assume that $b$ is the first such time.
 
 By replacing $a$ by an $a'<b$ sufficiently close to $b$, we may suppose that the mean curvature
 flow starting from $\partial K(a)$ remains smooth and compact
 for time at least $b-a$.  For $t\in [a,b]$, let $M(t)$ be the result of flowing $M(0)=\partial K(a)$
 for time $t-a$.
 Let $\Kup(t)$ be the closed region bounded by $M(t)$ such that
\[
  K(t)\subset \Kup(t).
\]
Since $K(b)\cap Z(b)$ is nonempty and since $K(b)\subset \Kup(b)$, we see that
$\Kup(b)\cap Z(b)$ is nonempty.  By Lemma~\ref{point-hitting-lemma} below, 
the first contact of $\Kup(t)$ and $Z(t)$
occurs at a point in $\partial \Kup(t)$, that is, a point in $M(t)$.
Thus $t\in [a,b]\mapsto M(t)$ is a smooth mean curvature flow that is disjoint from $Z$
at time $a$ but not at some later time, which violates~\eqref{equivalent-4}.
\end{proof}

\begin{lemma}\label{point-hitting-lemma}
Suppose that $U$ is an open subset of $N$ and that $p\in U$.
\begin{enumerate}[\upshape (1)]
\item For all sufficiently small $\eps>0$, there is a mean curvature flow
\[
t\in [0,\eps]\mapsto M(t)
\]
 of smoothly embedded, closed hypersurfaces in $U$
such that $p\in M(\eps)$.
\item Suppose $Z$ has Property~\eqref{equivalent-4} in
   Theorem~\ref{equivalent-definitions-theorem}. 
    If $T>0$ and if $p\in Z(T)$, then $Z(t)\cap U$ is nonempty
for all $t\le T$ sufficiently close to $T$.
\end{enumerate}
\end{lemma}

The second assertion implies that at the first time $Z(\cdot)$ bumps
into a smooth compact barrier, the contact occurs only at the boundary 
of the barrier.

\begin{proof}
Let $R>0$ be very small and let $q$ be point with $\dist(p,q)=R$.
Choose $R$ sufficiently small that the geodesic spheres $S_r:=\partial \BB(q,r)$ 
with $R/2 \le r \le 2R$ are smooth and compact and lie in $U$.
Let $\delta>0$ be such that, under mean curvature flow,
 each of those spheres remains smooth and compact and in $U$ during the time 
interval $[0,\delta]$. For $t\in [0,\delta]$, let $S_r(t)$ be the result of flowing $S_r$ for time $t$.
Choose $\eps\in (0,\delta]$ sufficiently small that $q$ lies in the region between $S_{R/2}(\eps)$
and $S_{2R}(\eps)$.  Thus there will be a unique $r\in (R/2,2R)$ such that $q\in S_r(\eps)$.
Now let $M(t)= S_r(t)$ for $t\in [0,\eps]$.  This proves \eqref{equivalent-1}.

To prove~\eqref{equivalent-2}, let $M(\cdot)$ and $\eps$ by as in \eqref{equivalent-1}.
Consider the mean curvature flow
\[
   t\in [T-\eps,T] \mapsto \Sigma(t):= M(t-T).
\]
Since $p\in Z(T)\cap \Sigma(T)$, 
we see that $Z(T-\eps)\cap \Sigma(T-\eps)$ is nonempty since $Z$ has Property~\eqref{equivalent-4}
in Theorem~\ref{equivalent-definitions-theorem}.
Thus $Z(T-\eps)\cap U$ is nonempty since $\Sigma(T-\eps)\subset U$.  
\end{proof}


\section{The Biggest flow}\label{biggest-flow-section}
\begin{theorem}\label{biggest-flow-theorem}
Let $N$ be a smooth Riemannian manifold and let $C$ be a closed subset of $N$. 
There exists a weak set flow $Y$ in $N\times [0,\infty)$, called the \textbf{biggest flow} generated by $C$, 
such that
\begin{enumerate}
\item $Y(0)=C$, and
\item If $Z$ is a weak set flow in $N\times [0,\infty)$
   with $Z(0)\subset C$, then $Z(t)\subset Y(t)$ for all $t\ge 0$.
\end{enumerate} 
\end{theorem}

\begin{proof}
Let 
\[
\mathcal{Z}:=\{Z\;\textrm{is a weak set flow with}\;Z(0) \subset C\}.
\]
Let  $Y$ be the closure of $\bigcup_{Z\in \mathcal{Z}}Z$. 

Note that $C\times \{0\}$ is an element of $\mathcal{Z}$.  Thus $C\subset Y(0)$.
Shrinking ball barriers (see Theorem~\ref{shrinking-ball-theorem}) imply that $Y(0)\subset C$.
Thus $Y(0)=C$.

It remains to check that $Y$ is indeed a weak set flow. 
By Theorem~\ref{equivalent-definitions-theorem},
 it suffices to check that if $t\in [a,b]\mapsto K(t)$ is a strong barrier with $a\ge 0$
and with $Y(a)\cap K(a)=\emptyset$, then $Y(t)\cap K(t)$ is empty for every $t\in (a,b)$.  
Choose $0<\eps<\dist(Y(a),K(a))$ sufficiently small that the spacetime set
\begin{equation}\label{K_eps_def}
K_\eps:=\{ (p,t): t\in [a,b], \, \dist(p,K(t)) \le \eps\}
\end{equation}

is a strong barrier.
  Let $Z\in \mathcal{Z}$.  Then 
\[
   \dist(Z(a),K(a))  \ge \dist(Y(a),K(a)) > \eps,
\]
so $Z(a)$ is disjoint from $K_\eps(a)$.  
Consequently, $Z(t)$ is disjoint from $K_\eps(t)$ for each $t\in [a,b]$.  That is,
\[
    \text{$\dist(Z(t), K(t)) > \eps$ for all $t\in [a,b]$}.
\]
Consequently,
\[
   \text{$\dist(Y(t), K(t))  \ge \eps$ for all $t\in (a,b)$}.
\]
This completes the proof that $Z$ is a weak set flow.
\end{proof}

\begin{definition}
If $C$ is a closed subset of $N$ and if $t\ge 0$, we let
\[
   F_t(C) = Y(t)
\]
where $Y\subset N\times [0,\infty)$ is the biggest flow generated by $C$.
\end{definition}

\begin{proposition}\label{semigroup-proposition}
$F_{t+s}(C)=F_t(F_s(C))$ for all $s,t\ge 0$.
\end{proposition}

\begin{proof}
Suppose that $t\in [0,\infty)\mapsto Y(t)$ is a weak set flow and that $T>0$.
Then $t\in [0,\infty)\mapsto Y(t-T)$ is a weak set flow.
Also, if $t\in [0,\infty)\mapsto Z(t)$ is a weak set flow and if $Z(0)\subset Y(T)$,
then 
\[
  t\in [0,\infty) 
  \mapsto
  \begin{cases}
  Y(t) &\text{if $t\in [0,T]$,} \\
  Z(t-T) &\text{if $t> T$}
  \end{cases}
\]
is a weak set flow.
(These facts follow easily from the definition of weak set flow.)
Theorem~\ref{semigroup-proposition} is an immediate consequence.
\end{proof}

We end this section by a characterization of the biggest flow in terms of solutions to the level set equation:
\begin{equation}\label{level_set_eq}
\pdf{u}t=  |\nabla u|\Div\left( \frac{\nabla u}{|\nabla u|} \right).
\end{equation} 

\begin{theorem}\label{containment-theorem}
Let $Y\subset N\times [0,T]$ be
let $\mathcal{U} \subset N\times [0,t]$ be an open set containing $Y$. 
Suppose that there exists a continuous function $u: \mathcal{U}\rightarrow \RR$ 
such that $u$ solves~\eqref{level_set_eq} in the viscosity sense,  and 
such that $Y_a:=u^{-1}(a)$ are compact for $a\in (-\eps,\eps)$ and such that $Y_0=Y$ . Then $Y(t)=F_t(Y(0))$ for every $t\in [0,T]$.     
\end{theorem}

\begin{proof}
Note that  $Y\subset \mathrm{Int}(\bigcup_{a\in (-\eps/4,\eps/4)}Y_a)$. Letting $\chi :[-1,1]\rightarrow \RR$ be a continuous function such that $\chi(x)=x$ for $|x|\leq \eps/4$ and $\chi(x)=1$ when $|x|\geq \eps/2$, 
the relabeling lemma \cite[3.2]{ilmanen-generalized} implies that $v:=\chi(u)$ is a solution of \eqref{level_set_eq} on $N\times [0,T]$. 
Now, \cite[6.3]{ilmanen-generalized} and 
Theorem~\ref{equivalent-definitions-theorem} imply that for each $a\in (-\eps/4,\eps/4)$, $Y_a$ is a weak set flow. In particular, $Y(t)\subset F_t(Y(0))$. 

Assuming that $Y(t)\neq F_t(Y(0))$ for some $t\in [0,T]$, set $t_0=\inf\{t\in [0,T]\;|\; Y(t)\neq F_t(Y(0))\}$. Note that Lemma \ref{shrinking-ball-corollary} implies that $Y(t_0)=F_{t_0}(Y(0))$ and that if $t>t_0$ is such that $t-t_0$ is sufficiently small, then 
\[
F_t(Y(0)) \subset \mathrm{Int}\left(\bigcup_{|a|<\eps/4}Y_a(t)\right).       
\]
By the definition of $t_0$, there exists some $t>t_0$ and $a\neq 0$ such that $Y_a(t)\cap F_t(Y(0))\neq \emptyset$. But as both $Y_a$ and $F_t(Y(0))$ are compact weak set flows with $F_{t_0}(Y(0))\cap Y_a(t_0)=\emptyset$, this contradicts Theorem \ref{avoidance-theorem}.  
\end{proof}

\section{Limits of Weak Set Flows}\label{limits-section}

\begin{definition}
Suppose $(Q,d)$ is a metric space.  We say that a sequence $Z_n\subset Q$ of closed subsets
{\bf Kuratowski-converges} to $Z\subset Q$ if
\[
   Z= \{ x: \limsup_n d(x,Z_n)=0\} = \{x: \liminf_n d(x,Z_n)=0\}.
\]
\end{definition}

Note that Kuratowski-convergence $Z_n\to Z$ is equivalent to: every point in $Z$ is a limit of a sequence 
of points $p_n\in Z_n$, and no point in $Q\setminus Z$ is a subsequential limit of such points.
Thus two metrics on $Q$ that give the same topology also give the same notion of Kuratowski-convergence.

If $Q$ is separable and if $Z_n$ is a sequence of closed subsets of $Q$, then, after passing to a subsequence,
$\dist(\cdot,Z_n)$ converges locally uniformly to a limit function $\delta(\cdot): Q\to [0,\infty]$ (by Arzela-Ascoli),
and thus the $Z_n$ Kuratowski-converge to $\{x: \delta(x)=0\}$.  If $Q$ is complete,
then $\delta(\cdot)=\dist(\cdot,Z)$.

\begin{theorem}\label{limits-theorem}
Let $g_n$ be a sequence of Riemannian metrics on $N$ that converge smoothly to a Riemannian metric $g$.
For $n=1,2,\dots$, let $Z_n\subset N\times [0,\infty)$ be a weak set flow (for the metric $g_n$)
such that the sequence $Z_n$ Kuratowski-converges to $Z$.
Then $Z$ is a weak set flow for the metric $g$.
\end{theorem}

Here the metric space is $N\times [0,\infty)$ with the spacetime metric 
\[
   d( (x_1,t_1), (x_2,t_2)) = \max\{\dist_g(x_1,x_2), \, |t_1-t_2|^{1/2}\}.
\]
Note in Theorem~\ref{limits-theorem} that $Z_n\to Z$ does not imply that $Z_n(t)\to Z(t)$ 
for each $t$.  For example, if $T_n\uparrow 1$, then the shrinking circles
\[
  Z_n: = \{ (p,t)\in \RR^2\times [0, T_n]:  \frac12 |p|^2 = T_n-t \}
\]
converge to the shrinking circle
\[
  Z:= \{ (p,t)\in \RR^2\times [0,1]:  \frac12 |p|^2 = 1-t \}
\]
But $Z_n(1)=\emptyset$ does not converge to $Z(1)=\{0\}$.

\begin{proof}[Proof of Theorem~\ref{limits-theorem}]
Let $a\ge 0$ and let $t:[a,b]\mapsto K(t)$ be a strong barrier with $K(a)$ disjoint from $Z(a)$.
By Theorem~\ref{equivalent-definitions-theorem},
 it suffices to show that $K(t)$ and $Z(t)$ are disjoint for each $t\in (a,b)$.

Fix a very small $\eps>0$, and let
\[
  K_\eps: t\in [a,b] \mapsto \{x\in N: \dist_g(x,K(t))\le \eps\}.
\]
In particular, we choose $\eps>0$ small enough so that $K_\eps$ is a strong barrier (with respect to $g$)
and such that $K_\eps(a)$ is disjoint from $Z(a)$.

For all sufficiently large $n$, $K_\eps$ is a strong barrier with respect to $g_n$
and $K_\eps(a)$ is disjoint from $Z_n(a)$.
Thus by Theorem~\ref{equivalent-definitions-theorem}, $K_\eps(t)$ is disjoint from $Z_n(t)$ for all $t\in[a,b]$,
so $\liminf\dist_{g_n}(K(t),Z_n(t)) \ge \eps$.  
Hence $\dist_g(K(t),Z(t)) \ge \eps$ for all $t\in (a,b)$.
In particular, $K(t)$ is disjoint from $Z(t)$ for all $t\in (a,b)$.
\end{proof}

\begin{remark}
The {\bf Kuratowski limsup} of a sequence of closed sets $Z_n$ in $Q$ is defined
to be $\{x:\liminf_n \dist(x, Z_n)=0\}$.  In Theorem~\ref{limits-theorem}, 
if we do not assume that the sequence $Z_n$ Kuratowski-converges,
 the Kuratowski limsup is a weak set flow.
The proof is exactly the same.
\end{remark}

\section{Boundaries}\label{boundary-section}

\newcommand{\Xx}{\mathcal{X}}

\begin{theorem}\label{boundary-theorem}
Suppose $C$ is a closed subset of a Riemannian manifold $N$.
Let
\[
  \Cc := \{(x,t): t\ge 0, \,\, x\in F_t(C) \}
\]
be the biggest flow generated by $C$, and let
\[
  M= \partial \Cc
\]
Then $M$ is a weak set flow and $M(0)=\partial C$.
\end{theorem}

In this section, for a subset $Q$ of $N\times [0,\infty)$, terms like ``interior'' and ``boundary''
refer to the relative topology in $N\times[0,\infty)$.  Thus $\interior(Q)$
is the largest subset of $Q$ that is relatively open in $N\times [0,\infty)$.
Correspondingly, 
\[
  \partial Q = \overline{Q}\setminus \interior(Q).
\]
For example, the interior of $N\times [0,\infty)$ is (in this context) all of $N\times[0,\infty)$
and therefore the boundary is the empty set.
Since the $\Cc$ in Theorem~\ref{boundary-theorem} is closed, $\overline{\Cc}=\Cc$, and thus $\partial \Cc=\Cc \setminus\interior(\Cc)$.

\begin{proof}
Let $\Ff$ be the family of all strong barriers $K:t\in [a,b]\mapsto K(t)$ such that
$K(a)\times \{a\}$ lies in the interior of $\Cc$. 
Let $W=\cup_{K\in \Ff}K$.
Since each $K\in \Ff$ is a weak set flow, $K\subset \Cc$ (by definition of biggest set flow)
and therefore $W\subset \Cc$.

Also, $W$ is a relatively open subset of $N\times[0,\infty)$.
This follows easily from the facts that if $t\in [a,b]\mapsto K(t)$ is a strong barrier, then
\begin{enumerate}[\upshape (i)]
\item $K$ can be prolonged to a strong barrier on a slighty longer time interval $[a,b+\eps]$, and
\item $t\in[a,b]\mapsto K_\eps(t)$ is a strong barrier for all sufficiently small $\eps>0$,
where $K_\eps$ is given by \eqref{K_eps_def}.
\end{enumerate}

Using strong barriers consisting of small shrinking balls (as in Theorem~\ref{shrinking-ball-theorem}),
one sees that
\begin{equation}\label{initial-containment}
   (\interior(C))\times \{0\} \subset W,
\end{equation}
and that
\[
   \interior(\Cc)\subset W.
\]
We have shown that $W$ is an open subset of $\Cc$ that contains $\interior(\Cc)$.
Thus
\[
  W = \interior(\Cc).
\]

By~\eqref{initial-containment}, $M(0)=\partial C$.  

It remains to show that $\partial \Cc$ is a weak set flow.
Let $t\in [a,b]\mapsto K(t)$ be a strong barrier with $a\ge 0$ such that $K(a)$ is disjoint from $M(a)$,
i.e., such that $K(a)\times\{a\}$ is disjoint from $\partial \Cc$.  We may assume that $K$ is connected.
Thus either
$K(a)\times\{a\}$ is disjoint from $\Cc$ or $K(a)\times\{a\}$ lies in 
$\operatorname{interior}(\Cc)=W$.
In the first case, $K$ is disjoint from $\Cc$ since $\Cc$ is a weak set flow.
In the second case, $K$ lies in $W=\operatorname{interior}(\Cc)$ by definition of $W$.
In either case, $K$ is disjoint from $\partial \Cc$.
Thus $\partial \Cc$ is a weak set flow by Theorem~\ref{equivalent-definitions-theorem}.
\end{proof}


\section{Mean curvature flow with a transport term}\label{X-flow-section}

Let $N$ be a smooth Riemannian manifold and let $X$ be a smooth vectorfield on $N$.
A smooth one-parameter family of hypersurfaces in $N$ is said to be an {\bf $X$-mean-curvature flow}
provided the normal component of velocity is everywhere equal to the mean curvature
plus the normal component of $X$.
If $t\in [a,b]\mapsto K(t)$ is a smooth barrier and if $x\in \partial K(t)$, we let
\begin{align*}
   \HH_K^X(x,t) &= \HH_K(x,t) + X\cdot\nu_K(x,t), \\
   \Phi_K^X(x,t) &= \vv_K(x,t) - \HH_K^X(x,t),
\end{align*}
and if $v$ is a tangent vector to $N$, we let
\[
    \Ricci^X(v,v) = \Ricci(v,v) + v\cdot \nabla_vX.
\]

We define a {\bf weak set flow for $X$-mean-curvature flow} (or {\bf weak $X$-flow} for short) 
by replacing $\Phi_K$ by $\Phi_K^X$
in Definition~\ref{weak-set-flow-definition}.
With three exceptions, all the theorems and proofs in this paper remain true provided
we make the following changes:
\begin{enumerate}
\item  Mean curvature flow, $\HH_K$, and $\Phi_K$ are replaced by $X$-mean curvature flow, 
   $\HH_K^X$, and $\Phi_K^X$.
\item Lower bounds of the form $\Ricci > \lambda$ (or $\Ricci\ge \lambda$)
     are replaced by $\Ricci^X>\lambda$ (or $\Ricci^X \ge \lambda$).
\item All of the global theorems in this paper assume that $N$ is complete with
  Ricci curvature bounded below.  In the case of $X$-flows, we add the 
  assumption that $|\nabla X|$ is bounded.
\end{enumerate}
The extra hypothesis  (3) ensures that compact surfaces remain compact under the flow,
and that the empty surface remains empty under the flow.  See Theorem~\ref{compact-theorem-X} below.

The three exceptions (in which there is something new in the statement and/or the proof)
are Theorems~\ref{finite-speed-theorem}, \ref{compact-theorem},
and \ref{long-time-theorem}.
However, with very slight modification, those results continue to hold
for $X$-mean-curvature flow:


\begin{theorem}[$X$-flow version of Theorem~\ref{finite-speed-theorem}]
\label{finite-speed-theorem-X}
For every $r>0$, $\lambda\in \RR$, and positive integer $n$, 
there is a constant $h=h(r,\lambda,n)>0$  with the following 
property.   Suppose that $N$ is a smooth Riemannian $n$-manifold,
that $R>r$, that the geodesic ball $\overline{\BB}(p,R)$ in $N$ is compact, 
and that the Ricci curvature of $N$
is $\ge \lambda$ on $\overline{\BB}(p,R)$.
If $t\in [0,\infty)\mapsto Z(t)$ is a weak set flow in $N$ and if
   $\dist(Z(0),p) > R$, then 
\begin{equation}\label{swango}
   \dist(Z(t),p) > R - (h+\chi)t \qquad\text{for all $t\in [0, (R-r)/(h+\chi)]$}, 
\end{equation}
provided $|X|\le \chi$ on $\overline{\BB}(p,R)$.
\end{theorem}

The proof is almost identical to the proof of Theorem~\ref{finite-speed-theorem}.

\begin{corollary}\label{finite-speed-corollary-X}
If $r\le \delta \le R$, then
\[
   \dist(p,Z(t)) > \delta  \quad\text{for all  $t \in [0, (R-\delta)/(h+\chi)]$}.
\]
\end{corollary}

\begin{theorem}[$X$-flow version of Theorem~\ref{compact-theorem}]
\label{compact-theorem-X}
Suppose that $N$ is a complete Riemannian manifold with Ricci curvature bounded below,
that $X$ is a smooth vectorfield on $N$ with $|\nabla X|$ bounded, 
and that $t\in [0,\infty)\mapsto Z(t)$ is a weak $X$-flow in $N$.
\begin{enumerate}
\item If $Z(0)$ is empty, then $Z(t)$ is empty for every $t$.
\item If $Z(0)$ is compact, then $\cup_{t\le T}Z(t)$ is compact 
for every $T<\infty$.
\end{enumerate}
\end{theorem}

\begin{proof}[Proof of Theorem~\ref{compact-theorem-X}]
It suffices to consider the case that $N$ is connected.  Let $x_0$ be a point in $N$.
Since $|\nabla X|$ is bounded,
\[
    c: = \sup\frac{|X(\cdot)|}{\max\{1,\dist(\cdot,x_0)\}} <\infty.
\]
If $Q$ is a closed subset of $N$, let $t\in [0,\infty)\mapsto F_t^X(Q)$ denote the biggest weak $X$-flow with 
  $F_0^X(Q)=Q$.
  
Let $h=h(1,\lambda,n)$ be as in Theorem~\ref{finite-speed-theorem-X},
 where $\lambda$ is a lower bound for Ricci curvature on $N$.

Let $R\ge 2$.
On the ball $\overline{\BB}(x_0,2R)$, $|X|$ is bounded by $2Rc$.
Since $\dist(x_0,\emptyset)=\infty>2R$, 
\[
   \dist(x_0,F_t(\emptyset)) > R \qquad \text{for all $t\ge 0$ with $t\le \frac{R}{h+2Rc}$.}
\]
by Corollary~\ref{finite-speed-corollary-X} (with $r=1$ and $\delta=R$).  Since $R\ge 1$,
\[
   \frac{R}{h+2Rc} = \frac1{(h/R) + 2c} \ge \frac1{h+2c}.
\]
Thus if we let $\tau=1/(h+2c)$,
\[
   \dist(x_0,F_t(\emptyset)) > R \qquad\text{for all $t\le \tau$}.
\]
Since this holds for every $R\ge 1$, we see that $F_t^X(\emptyset)=\emptyset$ for all $t\in [0,\tau]$.
By iteration, $F_t^X(\emptyset)=\emptyset$ for all $t\ge 0$.

We now prove (2).
For $r\ge 0$, let $B_r=\overline{\BB}(x_0,r)=\{x: \dist(x,x_0)\le r\}$.
Suppose $a\ge 1$ and $\dist(p,x_0)\ge 4a$.
We now derive a lower bound on the first time $t$ (if there is one) such that $p\in F_t^X(B_a)$.

Define $R$ by $\dist(p,x_0)=a+3R$.  Thus $R\ge a \ge 1$.

Now
\[
   \BB(p,2R)\subset \BB(x_0, \dist(x_0,p) + 2R)) = \BB(x_0, a+5R) \subset \BB(x_0,6R).
\]
Thus $|X|$ is bounded above by $6Rc$ on $\BB(p,2R)$.

Now $B_a$ is disjoint from $\overline{\BB}(p,2R)$, so by Corollary~\ref{finite-speed-corollary-X},
\[
  \dist(p, F_t^X(B_a)) > R \qquad \text{for $t\le \frac{R}{h+6Rc}$}.
\]
Now
\[
   \frac{R}{h+6Rc} = \frac1{(h/R)+6c} \ge \frac1{h+6c}
\]
since $R\ge 1$.  Thus if $T:=1/(h+6c)$, then
\[
   p \notin \cup_{t\in [0,T]}F_t^X(B_a).
\]
Since this holds for all $p$ with $\dist(p,x_0)\ge 4a$, 
\[
   \cup_{t\in[0,T]}F_t^X(B_a) \subset B_{4a}.
\]
By iteration,
\[
  \cup_{t\in [0,kT]} F_t^X(B_a) \subset B_{4^ka}.
\]
\end{proof}


\begin{theorem}[$X$-flow version of Theorem~\ref{long-time-theorem}]
\label{long-time-theorem-X}
Suppose that $N$ is a complete, smooth Riemannian manifold
with Ricci curvature bounded below,
and that $X$ is a smooth vectorfield with $|\nabla X|$ bounded.
If $Y, Z\subset N\times [0,\infty)$ are weak $X$ flows with $Y(0)$ compact, then
 for every $t<\infty$,
 \[
     \dist(Y(t),Z(t)) \ge e^{\lambda t} \dist(Y(0),Z(0)),
 \]
where $\lambda$ is a lower bound for $\Ricci^X$.
\end{theorem}

(Note that $\Ricci^X$ is bounded below because $\Ricci^X$ and $\Ricci$ differ at each point
by at most $|\nabla X|$.)

\begin{proof}
Given Theorem~\ref{compact-theorem-X}, the proof
of Theorem~\ref{long-time-theorem-X} is just like the proof of Theorem~\ref{long-time-theorem}.
\end{proof}


\begin{theorem}\label{smooth-case-theorem}
Suppose that $N$ is a complete Riemannian manifold with Ricci curvature bounded below
and that $X$ is a smooth vectorfield with $|\nabla X|$ bounded. 
Suppose that $M$ is a smooth, closed, embedded hypersurface in $N$,
and that $t\in [0,T]\mapsto M(t)$ is a smooth $X$-mean curvature flow with $M(0)=M$
and $\cup_{t\in [0,T]}M(t)$ compact.
Then
\[
   F_t^X(M)=M(t)
\]
for $t\in [0,T]$, where $F_t^X(\cdot)$ is biggest $X$-flow.

If $M$ bounds a closed region $Q$, then $F_t^X(Q)$ is the corresponding closed region bounded by $M(t)$.
\end{theorem}

\begin{proof}
Trivially, $t\mapsto M(t)$ is a weak $X$-flow, so $M(t)\subset F_t^X(M)$ for all $t\le T$.
Thus it suffices to show that $F_t^X(M)\subset M(t)$.

For $s>0$, let $M_s= \{x\in N: \dist(x,M)=x\}$.
Let $\eps>0$ be sufficiently small that for all $s\in [0,\eps]$, 
there is a smooth $X$-mean-curvature-flow 
\[
   t\in [0,T] \mapsto M_s(t)
\]
with $M_s(0)=M_s$.  Let $s\in (0,\eps]$.
 For each $t\in [0,T]$, let $K_s(t)$ be the union of $M_s(t)$ and
the connected components of $N\setminus M_s(t)$ that do not contained $M(t)$.
Then $t\in [0,T]\mapsto K_s(t)$ is a weak $X$-flow, so
$F_t^X(M)$ and $K_s(t)$ are disjoint for all $t\in [0,T]$ by Theorem~\ref{long-time-theorem-X}.
Since this holds for all $x\in (0,T]$, we see that $F_t^X(M)\subset M(t)$ for all $t\in [0,T]$.
This completes the proof that $F_t^X(M)=M(t)$.

To prove the assertion about $Q$, let $K_s=\{x\in N: \dist(x,Q)\ge s\}$.
Choose $\eps>0$ sufficiently small that there a smooth, compact $X$-mean curvature flow
on the time interval $[0,T]$ with initial surface $\partial K_s$.  Let $t\in [0,T]\mapsto K_s(t)$
be the corresponding flow of regions.

By Theorem~\ref{long-time-theorem-X}, $F_t^X(Q)$ and $\partial K_s(t)$ are disjoint for all $t\in [0,T]$.
It follows that $t\in [0,T]\mapsto F_t^X(Q)\cap K_s$ is a weak $X$-flow that is empty at time $0$.
Thus it is empty for all $t\in [0,T]$ by~Theorem~\ref{compact-theorem-X}.  Since this holds for all $s\in (0,\eps]$,
we see that $F_t^X(Q)\subset Q(t)$ for all $t\in [0,T]$.  The reverse inclusion holds trivially
(since $t\mapsto Q(t)$ is a weak $X$-flow.)
\end{proof}

\section{$X$-mean-convex flows}\label{X-mean-convex-section}

\begin{theorem}\label{X-mean-convex-theorem}
Suppose that $N$ is a complete Riemannian manifold with Ricci curvature bounded below, and that $X$
is a smooth vectorfield with $|\nabla X|$ bounded.
Suppose that $Q$ is a closed region in $N$ bounded by a compact hypersurface $M$,
and suppose that $Q$ is strictly $X$-mean-convex, i.e., that $\overrightarrow{H}+X^\perp$ is nonzero
and points into $Q$ at each point of $M$.

Then there is a continuous time-of-arrival function $u:Q\to [0,\infty]$ such that 
for each $t\in [0,\infty)$,
\begin{align*}
&F_t^X(M) = \{x: u(x)=t\}, \\
&F_t^X(Q) = \{x: u(x)\ge t\}, \\
&\partial F_t^X(Q)=M(t), \\
&\text{$\cup_{\tau\le t}M(\tau)$ is compact}, 
\end{align*}
where $F_t^X(\cdot)$ denotes biggest $X$-flow.
\end{theorem}

\begin{proof}
Since $M$ is smooth and compact, there is an $\eps>0$ and a smooth $X$-mean curvature flow
\[
    t \in [0,\eps]\mapsto M(t)
\]
with $M(0)=M$. For $t\in [0,\eps]$, let $Q(t)$
be the closed region (corresponding to $Q$) in $N$ bounded by $M(t)$.
Note that
\begin{equation}\label{little-while}
   \text{$Q(t) \subset Q$ and $Q(t)\cap M=\emptyset$ for $0<t\le\eps$}
\end{equation}
by the smooth maximum principle.
By Theorem~\ref{smooth-case-theorem}, 
\begin{equation}\label{same}
\text{$F_t^X(M)=M(t)$ and $F_t^X(Q)=Q(t)$ for all $t\in [0,\eps]$.}
\end{equation}
By~\eqref{little-while} and~\eqref{same}, $F_t^X(Q)\subset Q$ for $t\in [0,\eps]$ and thus (since $F_\tau^X$ preserves
inclusion)
\begin{equation*}
  F_{\tau+t}^X(Q)\subset F_\tau^X(Q)  \quad \text{for all $t\in [0,\eps]$ and $\tau\ge 0$}.
\end{equation*}
By transitivity of inclusion, this implies
\begin{equation}\label{Qs-nested}
  F_{T}^X(Q)\subset F_t^X(Q) \quad\text{for all $T\ge t \ge 0$}.
\end{equation}

By Theorem~\ref{compact-theorem-X}, $\cup_{t\in[0,T]}F_t^X(M)$ is compact for all $T<\infty$,
and by~\eqref{little-while} and by avoidance (e.g., Theorem~\ref{long-time-theorem-X}),
\[  
   F_{\tau+t}^X(Q)\cap F_\tau^X(M) =\emptyset \quad\text{for all $t\in (0,\eps]$ and $\tau\ge 0$}.
\]
Hence by~\eqref{Qs-nested},
\begin{equation}\label{Qs-Ms-disjoint}
   F_\tau^X(Q)\cap F_t^X(M) = \emptyset \quad\text{for all $\tau>t\ge 0$}.
\end{equation}
In particular, 
\begin{equation}\label{Q-M-disjoint}
  F_\tau^X(M)\cap F_t^X(M) = \emptyset \quad\text{for all $\tau>t\ge 0$}
\end{equation}
since $F_t^X(M)\subset F_t^X(Q)$ for all $t$.

Now define $u:Q\to [0,\infty]$ by
\[
u(x) = \begin{cases}  
   t &\text{if $x\in F_t^X(M)$,} \\
   \infty &\text{if $x\in Q \setminus \cup_tF_t^X(M)$}.
   \end{cases}
\]
(This is well-defined since the $F_t^X(M)$ are disjoint.)

Since the $F_t^X(M)$ trace out a closed subset of spacetime, $u:Q\to [0,\infty]$ is continuous.

Now suppose that $x\in Q\setminus F_T^X(Q)$.
Then the spacetime set
\begin{equation}\label{the-set}\tag{*}
   \{(x,t): x\in F_t^X(Q), \, t\ge 0\},
\end{equation}
contains the point $(x,0)$ but not the point $(x,T)$.  Thus there is a $t\in [0,T)$ such that $(x,t)$
lies in the boundary $\mathcal{B}$ (relative to $N\times [0,\infty)$) of the set~\thetag{*}.
By Theorem~\ref{boundary-theorem}, $\mathcal{B}$ is a weak $X$-flow starting from $M$.  
Thus $\mathcal{B}$ lies in the biggest
 such weak $X$-flow, so
\[
    (x,t) \in \mathcal{B} \subset  \{  (y,\tau): y\in F_\tau^X(M), \, \tau\ge 0\},
\]
and therefore $x\in M(t)$.  Hence we have shown 
\[
   x\in Q\setminus F_T^X(Q) \implies u(x) < T.
\]
On the other hand
\[
  u(x)< T \implies x\in M_{u(x)} \implies x\notin F_T^X(Q)
\]
by~\eqref{Qs-Ms-disjoint}.  Thus $F_T^X(Q)=\{u\ge T\}$.

Finally, if $t>0$, then every point in $\{u=t\}$ is a limit of points in $\{u<t\}$ by Corollary~\ref{shrinking-ball-corollary}),
so $\{u=t\}$ has no interior.  Hence $\{u=t\}$ is the boundary of $\{u\ge t\}$.
\end{proof}

\begin{remark} 
Theorem~\ref{X-mean-convex-theorem} 
 remains true (with the same proof) for any closed set $Q$ and for $M=\partial Q$ (not necessarily smooth)
such that
\[
   F_t^X(Q) \subset Q\setminus M  
\]
for all $t$ is some small time interval $(0,\eps]$.  
\end{remark}

\section{Varifold flows}\label{Brakke_sec}

An $m$-dimensional {\bf integral Brakke $X$-flow} in a Riemannian manifold $N$
is a one-parameter family $t\in[0,\infty)\mapsto M(t)$ of Radon measures on $N$
such that for almost every $t$, $M(t)$ is the radon measure associated to
an $m$-dimensional integral varifold in $N$, and such that for every $C^2$, nonnegative, compactly
supported function $\phi$ on $N\times[0,\infty)$,
\newcommand{\barD}{\overline{\operatorname{D}}}
\begin{equation}\label{basic-X-brakke}
\barD_t \int \phi\,dM(t)
\le
\int \left(
   \pdf{\phi}t 
   +
   \nabla \phi^\perp\cdot (X + H)  
   - \phi H\cdot (H + X^\perp) \right) \,dM(t),
\end{equation}
where $\bar{D}_tf(t):=\limsup_{h\to 0}(f(t+h)-f(t))/h$.
As in the case of Brakke flow, the inequality~\eqref{basic-X-brakke} follows from the 
special case when $\phi$ is independent of time; see~\cite{Ilm_elip}*{\S6} or~\cite{Bra}*{3.5}.
Also, as for Brakke flow, the right side of \eqref{basic-X-brakke} should be interpreted as $-\infty$
if any terms in the the expression do not make sense at time $t$; see the discussion
in~\cite{Ilm_elip}*{\S6}.

For integral varifolds, $H=H^\perp$, so we can rewrite~\eqref{basic-X-brakke} as
\begin{equation}\label{basic-X-brakke-2}
\begin{aligned}
\barD_t \int &\phi\,dM(t) \\
&\le
\int \left(
   \pdf{\phi}t 
   +
   \nabla \phi^\perp\cdot X + \nabla \phi\cdot H  
   - \phi H\cdot X  - \phi |H|^2 \right) \,dM(t)
\\
&=
\int \left(
   \pdf{\phi}t 
   +
   \nabla \phi^\perp\cdot X - \Div_M \nabla \phi 
   + \Div_M (\phi X)  - \phi |H|^2 \right) \,dM(t)
\\
&=
\int \left(
   \pdf{\phi}t 
   +
   \nabla \phi^\perp\cdot X - \Div_M \nabla \phi    \right. 
   \\
   &\qquad\qquad \left.\vphantom{\frac12} +\, (\nabla\phi)^{\rm tan}\cdot X + \phi\Div_MX  - \phi |H|^2 \right) \,dM(t)
\\
&=
\int \left(
   \pdf{\phi}t 
   +
   \nabla \phi\cdot X - \Div_M \nabla \phi 
    + \phi\Div_MX  - \phi |H|^2 \right) \,dM(t).
\end{aligned}
\end{equation}

\begin{theorem}\label{Brakke_avoid_thm}
Let $t\in [0,\infty)\mapsto M(t)$ be an $m$-dimensional integral  $X$-Brakke flow
in a smooth $(m+1)$-dimensional Riemannian manifold $N$.
Let $Z\subset N\times[0,\infty)$ be the spacetime support of the flow (i.e., the closure
 in $N\times \RR$ of  $\cup_t(\spt M(t))\times \{t\})$).
Then $Z$ is a weak $X$-flow.
\end{theorem}

Theorem \ref{Brakke_avoid_thm} was proved in \cite{Ilm_elip}*{10.5} for Brakke flows in $\RR^{n+1}$ (with $X=0$).

\begin{proof}[Proof of Theorem~\ref{Brakke_avoid_thm}]
Let 
\[
   t\in [a,b]\subset [0,\infty)\mapsto K(t)
\]
be a strong barrier (as in~\S\ref{equivalent-definitions-section})
 such that $K(t)$ is disjoint from $Z(t)$ for $t\in [a,b)$.
By Theorem \ref{equivalent-definitions-theorem}, it suffices to show that $Z(b)$ is disjoint from $K(b)$.

Let $r(\cdot,t)$ be the signed distance to $\partial K(t)$ such that $r$ is positive in the complement of $K(t)$.
Then for $x\in \partial K(t)$, 
\begin{equation*}
0
>\Phi_K 
= \vv_K - H^X_K 
= -\pdf{r}t + \Delta r - X\cdot \nabla r
\end{equation*}
by~\eqref{f-expressions} with $r$ in place of $f$.
Consequently, we can choose $\delta>0$ sufficiently small and $k>0$ so that
wherever $|r|\le \delta$, 
the function $r$ is smooth and  
\begin{equation}\label{heat-inequality}  
\pdf{r}t - \Delta r + X\cdot \nabla r \ge k. 
\end{equation}
We also choose $\delta$ to be less that $\dist(Z(a),K(a))$.

By~\eqref{basic-X-brakke-2},
\begin{equation}\label{basic-without-Hsquared}
\barD_t\int \phi\,dM(t)
\le
\int \left(
   \pdf{\phi}t 
   +
   \nabla \phi\cdot X - \Div_M \nabla \phi 
    + C\phi \right) \,dM(t),
\end{equation}
where $C$ is $m$ times the maximum of $|\nabla X|$ on a compact set containing the support of $\phi$.

Now let $\phi = ((\delta - r)^+)^3$.  Note that this function is $C^2$ on the points of $N\times[a,b]$ in the support of the flow.
Letting $s=(\delta-r)^+$, we have
\begin{align*}
 \pdf{\phi}t &= - 3s^2 \pdf{r}t, \\
 \nabla\phi &= - 3s^2 \nabla r \\
 \nabla^2\phi &= - 3s^2 \nabla^2r + 6s\nabla r\otimes\nabla r, \\
 \Div_M(\nabla \phi) 
 &= -3s^2\Div_M\nabla r + 6s|(\nabla r)^{\rm tan}|^2 \\
 &= -3s^2(\Delta r - \nabla^2r(\nn,\nn)) + 6s(1 - |\nn\cdot\nabla r|^2),
\end{align*}
 where $\nn(x,t)$ is a unit normal to the approximate tangent plane to $M(t)$ at $x$.
Thus by~\eqref{basic-without-Hsquared} and~\eqref{heat-inequality},
\begin{align*}
\barD_t\int\phi\,dM(t)
&\le
\int \left( -3s^2\left( \pdf{r}t - \Delta r + \nabla r\cdot X  + \nabla^2r(\nn,\nn)  \right) \right.
\\
&\qquad\qquad
   - \left. \vphantom{\frac12} 6s(1 - |\nn\cdot\nabla r|^2)  + C s^3 \right) \,dM(t)
   \\
&\le
\int \left( -3s^2 k  +  3s^2| \nabla^2r(\nn,\nn)|   -  6s(1 - |\nn\cdot\nabla r|^2)  + C\delta s^2  \right) \,dM(t)
\\
&\le
\int \left(  3s^2|\nabla^2r(\nn,\nn)| - 6s(1 - |\nn\cdot\nabla r|^2) \right) \, dM(t)
\end{align*}
provided we choose $\delta< 3k/C$.  Now $\nabla^2r(\cdot,\cdot)$ is a quadratic form that vanishes
on $\nabla r$, so
\[
    |\nabla^2r(\nn,\nn)| \le c(1 - (\nn\cdot\nabla r)^2)
\]
for some constant $c$.  Thus
\begin{align*}
D_t\int \phi\,dM(t)
&\le
\int   (3s^2 c - 6s)(1 - |\nn\cdot\nabla r|^2)  \, dM(t)
\\
&\le
\int 3s   (\delta c - 2)(1 - |\nn\cdot\nabla r|^2)  \, dM(t),
\end{align*}
which is $\le 0$ provided we chose $\delta< 2/c$.

Since $\int \phi\,dM(t)$ is nonnegative, zero at the initial time $a$, and decreasing, 
it is zero for all $t\in [a,b]$.
Thus $\dist(Z(t),K(t))\ge \delta$ for all $t\in [a,b]$.
\end{proof}

\appendix
\renewcommand{\thetheorem}{\thesection\arabic{theorem}}
\setcounter{theorem}{0}

\section{}

\begin{theorem}\label{app_thm}
Suppose that $X$ and $Y$ are closed subsets of $N$ such that
\[
  r:=\frac12\dist(X,Y) > 0
\]
and such that
\[
   \{p: \dist(p,X)= r\}
\]
is compact.
Then there is a compact, $C^1$ embedded hypersurface surface $M$ separating $X$ and $Y$ such
that
\[
   \dist(X,M)=\dist(Y,M)=r.
\]
\end{theorem}

\begin{proof}
Note that there is a  $\delta\in (0,r)$ such that 
\begin{equation}\label{between-region}
  \{p: r-\delta\le \dist(p,X) \le r+\delta\}
\end{equation}
is compact.   By replacing $\delta$ be a smaller $\delta$, we can assume
that geodesic balls with centers in~\eqref{between-region} and with radii $\le \rho$
have smooth boundaries.

Let
\begin{align*}
   X' &= \{p: \dist(p,X)\le r-\delta\}, \\
   Y' &= \{p: \dist(p,X)\ge r+\delta\},
\end{align*}
and let
\begin{align*}
A&= \{p: \dist(p,X)\le r\} = \{p:\dist(p,X')\le \delta\}, \\
B&= \{p: \dist(p,Y') \le \delta\}, \\
Z&= A\cap B, \\
U&= N\setminus(A\cup B).
\end{align*}
Note that $\overline{U}$ is compact.

Consider a point $z\in Z$.  Let $C^z_1$ and $C^z_2$ be shortest geodesics joining $z$ to $X'$ and to $Y'$.
Then $C^z_1\cup C^z_2$ is a shortest curve  joining $X'$ to $Y'$, and thus is a geodesic.
Consequently, $C^z_1$ and $C^z_2$ are unique and depend continuously on $z\in Z$.
Therefore 
\[
   z\in Z\mapsto \vv(z)
\]
 is continuous, where $\vv(z)$ is the unit tangent vector to $C^z_1\cup C^z_2$
at $z$ that points out of $C^z_1$ and into $C^z_2$.

Let $h: U \to \RR$ be the function that minimizes $\int |Dh|^2$ subject to
\begin{align*}
h&=-1 \quad \text{on $(\partial A)\setminus B$ and} \\
h&=1 \quad \text{on $(\partial B)\setminus A$}.
\end{align*}
Then $h$ is harmonic (and therefore smooth) on $U$
and continuous
 on $\overline{U} \setminus Z$.
 
(The continuity holds because if $p\in \partial{U}$ and if $q$ is a point in $X'\cup Y'$
closest to $p$, then $\BB(q,\delta)\subset U^c$ and $p\in \partial \BB(q,\delta)$.)

Let $c\in (-1,1)$ be a regular value of $h$, and let
\[
  M = h^{-1}(c)\cup Z.
\]
To prove that $M$ is $C^1$, it suffices to show that if $p_i\in M\cap U$ converges to $p\in Z$,
then
\[
  \frac{\nabla h(p_i)}{|\nabla h(p_i)|} \to \vv(p).
\]
Let $\BB(q_i,r_i)$ be the largest ball in $U$ that contains $p_i$.
We work in normal coordinates at the point $p$.
Let
\[
   U_i = (U - q_i)/r_i
\]
and
\begin{align*}
  &h_i: U_i \to \RR, \\
  &h_i(x) = h(r_i(q_i+x)).
\end{align*}
Note that $U_i$ converges to the slab
\[
  \{ x\in \RR^{n+1}:  0 < x\cdot\vv(p) < 1\}.
\]
Therefore $h_i$ converges smoothly
to the harmonic function
\[
   x\cdot\vv(p)
\]
and $p_i$ converges (perhaps after passing to a subsequence)
to a point $p'$ such that $p'\cdot \vv(p)=c$.
The result follows immediately. 
\end{proof}

\bibliography{HershkovitsWhite}
\bibliographystyle{alpha}

\end{document}